\documentclass{jams-l}
\newcommand{\R}{\mathbb{R}}
\newcommand{\E}{\mathbb{E}}

\usepackage{mathrsfs,mathtools}
\usepackage{enumerate}
\usepackage{graphicx}
\usepackage{float}
\usepackage{xcolor}
\usepackage[colorlinks=true,allcolors=blue,backref=page]{hyperref}  
\usepackage[noabbrev,capitalize,nameinlink]{cleveref}

\definecolor{darkred}{RGB}{200,0,0}
\hypersetup{colorlinks={true},linkcolor={darkred},citecolor={darkred}}  

\newtheorem{theorem}{Theorem}
\newtheorem*{theorem*}{Theorem}
\newtheorem{lemma}[theorem]{Lemma}

\newtheorem*{defn}{Definition}
\newtheorem{claim}[theorem]{Claim}

\newtheorem*{observation*}{Observation}

\newtheorem*{remark*}{Remark}
\newtheorem*{claim*}{Claim}

\newcommand{\cK}{\mathcal{K}}

\DeclareMathOperator{\coD}{\mathnormal{co}-\mathnormal{D}}

\renewcommand{\S}{\mathbb{S}}

\newcommand{\cH}{\mathcal{H}}
\newcommand{\cP}{\mathcal{P}}

\newcommand{\dist}{\mathsf{dist}}
\newcommand{\eps}{\varepsilon}
\newcommand{\ip}[2]{\langle #1,#2 \rangle}
\DeclareMathOperator{\conv}{\mathsf{conv}}

\newcommand{\N}{\mathbb N}

\numberwithin{equation}{section}

\begin{document}

\title{Approximation depth of convex polytopes}


\author[Bakaev]{Egor Bakaev}
\address{Department of Computer Science, the University of Copenhagen}

\author[Brunck]{Florestan Brunck}
\address{Department of Computer Science, the University of Copenhagen}

\author[Yehudayoff]{Amir Yehudayoff}
\address{Department of Computer Science, the University of Copenhagen,
and Department of Mathematics, Technion-IIT}

\begin{abstract}
We study approximations of polytopes in the
standard model for computing polytopes using Minkowski sums and (convex hulls of) unions. 
Specifically, we study the ability to approximate a target polytope by polytopes of a given depth. Our main results imply that simplices can only be ``trivially approximated''. On the way, we obtain a characterization of simplices as the only ``outer additive'' convex bodies. 
\end{abstract}

\maketitle

\section{Introduction}

There are numerous classical approximation problems in convex geometry, such as approximating a given convex set by polytopes and approximating a given polytope by specific polytopes. 
These problems, in their many forms, have been studied for many years with various types of motivation (see~\cite{boroczky2000approximation, gruber1983approximation,gr2003unbaum,kallay1982indecomposable} and references within).

Shephard and others, for example, considered the following approximation problem (see~\cite{shephard1964approximation} and references within).
The Minkowski sum of two sets $K, L$ in Euclidean space is defined as $K+L = \{x+y : x \in K, y\in L\}$. 
A target convex set $T$ is approximated by a collection $\cK$ of convex sets
if there is a sequence of sets $(K_j)$ such that each $K_j$ is a finite sum of elements of $\cK$ so that $K_j \to T$ as $j \to \infty$
in the Hausdorff metric. 

As we describe next, Shephard proved that an indecomposable~$T$ can only be ``trivially approximated''. 
A set $K \subset \R^n$ is indecomposable if it is ``prime with respect to Minkowski summation'' (see definition below).
A homothet of $K \subset \R^n$
is a set of the form $H=x+\lambda K$ where $x \in \R^n$ and $\lambda  \geq 0$ is called the coefficient of homothety. 
Denote by $\cH_K$ the set of all homothets of $K$. 

\begin{defn}
A set $K \subset \R^n$ is indecomposable if for every $L_1,\ldots,L_m \subset \R^n$ such that $K = L_1+\ldots+L_m$, it holds that $L_1,\ldots,L_m \in \cH_K$.
\end{defn}

Indecomposable polytopes were studied in many works (see e.g.~\cite{firey1964addition,shephard1964approximation,shephard1963decomposable}). 
It is known, e.g, that if $K$ is a polytope so that all of its two-dimensional faces are triangles then $K$ is indecomposable~\cite{shephard1963decomposable}.
Indecomposability was also recently used to study the expressivity of monotone ReLU neural networks and input convex neural networks (ICNNs);
see~\cite{bakaev2025depth,valerdi2024minimal}.
Returning to the approximation problem mentioned above, Shephard proved that an indecomposable $T$ can be approximated by the class $\cK$ only if $\cK$ contains a homothet of~$T$.

\subsection{Constructing polytopes}
We study a non-asymptotic version of the approximation problem, which is motivated by the expressivity of neural networks (see~\cite{bakaev2025depth,haase2023lower,hertrich2021towards} and references within).
We consider the following model for constructing polytopes. The two operations are: Minkowski sum and union. 
For two sets $L,K \subset \R^n$, let
\[K*L = \conv (K \cup L)\]
where $\conv$ denotes convex hull.
Polytopes of depth zero are points:
\[\cP_{n,0} = \{ \{x\} : x \in \R^n\}.\]
Polytopes of depth $d > 0$ are defined inductively as
\[\cP_{n,d} = \left\{ (K_1*L_1) + \ldots + (K_m*L_m) : m \in \N, K_j,L_j \in \cP_{n,d-1} \right\}.\]
The standard name for polytopes in $\cP_{n,1}$ is zonotopes.

We can now define the depth complexity of a polytope $P \subset \R^n$ as the minimum $d$ so that $P \in \cP_{n,d}$.
The depth complexity of a polytope is always finite. In fact, if $P$ has $m$ vertices then
its depth complexity is at most $\lceil \log_2 m \rceil$. This bound is sometimes not sharp; e.g., the depth complexity of cubes is one.

This model has a one-to-one correspondence with the following neural network model for computing convex functions.
The input gates compute linear functions on $\R^n$. Inner gates compute functions of the form $g(x) =  \sum_j c_j \max\{a_j(x),b_j(x)\}$
where for every $j$, it holds that $c_j > 0$ and $a_j,b_j$ were previously computed. 
A network can thus be represented as a directed acyclic graph. 
The number of hidden layers in the network is the maximum number of inner gates in an input-output directed path. 

The correspondence between polytopes and convex functions is seen via support functions. 
Every closed convex set $P \subset \R^n$
defines a support function $h_P : \R^n \to \R$
via
\[h_P(x) = \max \{ \ip{x}{y} : y \in P \}.\]
Support functions have many important properties (see e.g.~\cite{gr2003unbaum}).
In particular, there is a one-to-one correspondence between  $P$ and $h_P$.
This correspondence translates to the computational setting: $P \in \cP_{n,d}$ iff $h_P$ is computed by a neural network with $d$ hidden layers.
There is a long sequence of related works on the ``functional side'' studying monotonicity in the context of neural networks; see~\cite{bakaev2025depth,daniels2010monotone,mikulincer2024size,sill1997monotonic,sivaraman2020counterexample} and references within.

We are interested in the depth complexity of {\em approximating} a given polytope.
We measure the quality of an approximation using two measures of distance (defined below). 
As opposed to Shephard's result, in our non-asymptotic setting, some indecomposable polytopes can be approximated non-trivially.
The reason is that, in dimension $n \geq 3$, the set of indecomposable polytopes is dense~\cite{schneider2013convex}, so every polytope $P \subset \R^n$ is close to an indecomposable polytope $Q \subset \R^n$
(for example, we can build $Q$ that is close to $P$ such that all of its two-dimensional faces are triangles).
If we start with a zonotope $P \in \cP_{n,1}$, then we get an indecomposable $Q$ that can be easily approximated. 

A particularly important polytope in this setting is the simplex.
A standard geometric representation of the $n$-dimensional simplex $\Delta^n$ is in $\R^n$ as the orthogonal projection of $\conv \{e_1,\ldots,e_{n+1}\}$ to the space orthogonal to $(1,1,\ldots,1) \in \R^{n+1}$,
where $e_1,\ldots,e_{n+1} \in \R^{n+1}$ are the standard unit vectors. 
For concreteness, in this work $\Delta^n$ refers to this geometric representation of the simplex.
Most of our results, however, hold for arbitrary simplices (which are affine images of the geometric simplex $\Delta^n$).

The depth complexity of the simplex $\Delta^n$ is known to be $d = \lceil \log_2 (n+1) \rceil$. The upper bound is trivial
and the lower bound was proved by Valerdi~\cite{valerdi2024minimal}.
The depth complexity of the simplex is particularly important because if the depth complexity of $\Delta^{n}$ is $d$ then every piecewise linear function on $\R^{n-1}$ can be computed with $d$ hidden layers~\cite{arora2018understanding,wang2005generalization}.

It is well known that simplices are indecomposable (e.g.,~\cite{shephard1963decomposable}).
Our main finding is the stronger statement that simplices can only be ``trivially approximated''. For example, we show that if for some depth parameter $d$, there is a sequence of polytopes $(P_j)$ in $\cP_{n,d}$ such that $P_j \to \Delta^n$ as $j \to \infty$ then $d \geq \lceil \log_2 (n+1) \rceil$.

\subsection{Distances} We measure the quality
of approximations through the following two notions of distance. 
We start with the in-out distance.
For $H \in \cH_K$, denote by $\lambda(H) = \lambda_K(H)$ the coefficient of homothety of $H$.
Define outer and inner coefficients as follows.
For $L ,K \subset \R^n$,
the outer 
coefficient of $L$ with respect to $K$ is\footnote{It is $\infty$ when the set of $H$'s is empty.}
\[ \lambda_{K,o}(L) = \inf \{ \lambda_K(H): H \in \cH_K, H \supseteq L\}.\]
The inner coefficient of $L$ with respect to $K$ is
\[\lambda_{K,i}(L) = \sup \{ \lambda_K(H): H \in \cH_K, H \subseteq L\}.\]
A set $L \subset \R^n$ is called a convex body if it is convex, compact and has non-empty relative interior.\footnote{In some works, the term ``body'' means non-empty interior.}
If $L \subseteq \R^n$ is a convex body then $0 < \lambda_{\Delta,o}(L) < \infty$.

\begin{defn}
The in-out distance is defined by\footnote{When the expression is $\infty-\infty$, it is not defined.}
\[D_{io}(L;K)
= \lambda_{K,o}(L) - \lambda_{K,i}(L).\]
\end{defn}
It is somewhat similar to the Banach--Mazur distance, but it is not invariant under scaling or invertible linear maps. It is, in fact, positively homogeneous: for all $\lambda > 0$,
\[D_{io}(\lambda L;  K)
= \lambda D_{io}(L;K).\]
Consequently, the in-out distance depends on the scale of $L$. The smaller $L$ is, the smaller $D_{io}(L;K)$ is.

The second notion of distance, the empty corners distance, is defined only for the simplex. Consider the simplex $\Delta = \Delta^{n}$.
For $H \in \cH_\Delta$, the vertices of $\Delta$ are in one-to-one correspondence to the vertices of~$H$, so we can identify them and denote them by
$V(\Delta) = V(H)$.
For a set $L \subset \R^n$,
for $H \in \cH_\Delta$ so that $L \subseteq H$ and for $v \in V(\Delta)$, the size of the empty corner of $L$ with respect to $v$ is 
\[E_v(L;H)= \sup \big\{ \lambda_\Delta(F): F \in \cH_\Delta, v \in V(F),
 F \subseteq H, F \cap L = \emptyset \big\};\]
it is zero when the set of $F$'s is empty.
The total empty corner size is 
$$E(L;H) = \sum_{v \in V(\Delta)} E_v(L;H).$$
For a bounded $L \subset \R^n$, let $\Delta_o(L)$ denote the unique $H \in \cH_{\Delta^n}$
so that $H \supseteq L$ and $\lambda_{\Delta}(H) = \lambda_{\Delta,o}(L)$.
In other words, it is the smallest homothet of $\Delta$ containing $L$. 
Uniqueness holds
because if $H,H' \in \cH_\Delta$ are such that $H \cap H' \neq \emptyset$ then $H \cap H' \in \cH_\Delta$. 

\begin{defn}
The empty corner distance of 
a convex body $L \subset \R^n$ is
\[D_e(L) = D_e(L;\Delta^n) = \frac{E(L;\Delta_o(L))}{\lambda_{\Delta,o}(L)}.\]
\end{defn}
The empty corner distance is invariant to scale: for all $\lambda >0$,
\[D_e(L) = D_e(\lambda L).\]
It trivially holds that\footnote{In fact, the bound $D_e(L) \leq n - \frac{1}{n}$ is always true with equality iff $L = -\Delta$. Indeed, assume that $\lambda_{\Delta,o}(L) =1$. If $U \subset V(\Delta)$ is of size $|U|=n$, then $\sum_{v \in U} E_v(L;\Delta) \leq n-1$ because there is at least one point of $L$ on the corresponding facet (see \Cref{lem:um-ratios}). So $D_e(L) \leq (n+1)(n-1)/n$. This calculation is sharp only for $L = - \Delta$ because the matrix that has zeros on the diagonal and ones off the diagonal has full rank
(so in the equality case, all the $E_v(L;\Delta)$'s are the same).
} $0 \leq D_e(L) \leq n+1$.
It is sometimes convenient to work with the co-distance
\[\coD_e(L) = \coD_e(L;\Delta^n) := n+1 - D_e(L).\]

The following simple observation partially justifies the term ``distance''. 

\begin{observation*}
For every convex body $L \subset \R^n$,
\[L \in \cH_\Delta \ \iff \ 
D_{io}(L;\Delta^n) = 0 \ \iff \ 
D_e(L;\Delta^n) = 0.\]
\end{observation*}

\subsection{Results}
We start with a couple of theorems showing that simplices are stably indecomposable.  
The theorems show that the Minkowski sum operation cannot decrease the distance to a simplex. The theorems are stronger than Shephard's inapproximability result mentioned above for the special case of simplices. 

\begin{theorem}
\label{thm:Dio}
If $L = L_1+\ldots+L_m$ where $L_1,\ldots,L_m \subset \R^n$ are convex bodies then
$$D_{io}(L;\Delta^{n}) \geq \max \{ D_{io}(L_j;\Delta^{n}) : j \in [m]\}.$$
\end{theorem}

\begin{theorem}
\label{thm:De}
If $L = L_1+\ldots+L_m$ where $L_1,\ldots,L_m \subset \R^n$ are convex bodies then
\[\coD_e(L) \leq  \max \{ \coD_{e}(L_j) : j \in [m]\}.\]
\end{theorem}

The following theorem describes the effect of the union operation on the empty corner distance.

\begin{theorem}
\label{thm:DeUnion}
If $L = L_1*L_2$ where $L_1,L_2 \subset \R^n$ are convex bodies then
\[\coD_e(L) \leq
\coD_e(L_1) + \coD_e(L_2).\]
\end{theorem} 

Next, we calculate the distance of the simplex from centrally symmetric sets.
A set $L \subseteq \R^n$ is centrally symmetric if there is $p \in \R^n$ so that for all $x \in \R^n$ we have $p+x \in L$ iff $p-x \in L$.
The point $p$ is called the center of symmetry of $L$. 
It is worth mentioning that some works measured how far is a given convex body from being  centrally symmetric using inscribed simplices (see e.g.~\cite{toth2009asymmetry} and references within).
Surprisingly, the distance is the same for all centrally symmetric sets. 

\begin{lemma}
\label{lem:DeSym}
	If $L \subset \R^n$ is a centrally symmetric convex body then $D_e(L) = n-1$.
\end{lemma}

We can now draw the following conclusion that says that simplices can only be trivially approximated. The next theorem fully describes the best approximation of the simplex by depth-$d$ polytopes.  
For a depth parameter $d>0$,
let\footnote{$P \not \in \cP_{n,0}$ because points are not convex bodies. As the theorem shows, the minimum is attained.}
\[D_e(\cP_{n,d}) := \min \{D_e(P;\Delta^n):P \in \cP_{n,d} \setminus \cP_{n,0} \} .\]

\begin{theorem}
\label{thm:Simplex} For $0 < d < \lceil \log_2(n+1) \rceil$,
\[ D_e(\cP_{n,d}) = n+1- 2^d.\]
\end{theorem}

\Cref{thm:Simplex} is stronger than the $\lceil \log_2(n+1) \rceil$ lower bound on the depth complexity of $\Delta^n$ proved in~\cite{valerdi2024minimal}.
It says that we cannot even approximate $\Delta^n$
in depth smaller than $\lceil \log_2 ( n+1) \rceil$.
For example,
if the polytope $P$ has depth complexity $d < \lceil \log_2 (n+1) \rceil$ then $D_e(P) \geq 1$.

The theorem also allows to move from inapproximability on the ``polytope side''
to inapproximability on the ``function side''.
The Blaschke selection theorem 
says that the collection of convex bodies is locally compact. It implies that for every $n$ and $\delta > 0$,
there is $\eps>0$ such that if $D_e(P) \geq \delta$ then
 \[\E |h_P(X) - h_{\Delta^n}(X)| \geq \eps,\]
where $X \sim \S^{n-1}$ is sampled from the uniform (Haar) probability measure on the sphere.
In particular, there is $\eps_0 = \eps_0(n) > 0$ such that if the polytope $P$ has depth complexity $d < \lceil \log_2 (n+1) \rceil$ then 
 \[\E |h_P(X) - h_{\Delta^n}(X)| \geq \eps_0.\]
It is worth noting that $\eps_0 \to 0$ as $n \to \infty$. 
Indeed, $\E |h_{\Delta^n}(X)| \leq O(\tfrac{\log n}{\sqrt{n}})$ 
so for $P = \{0\} \in \cP_{n,0}$,
\[\E |h_P(X) - h_{\Delta^n}(X)| \to 0 \]
as $n \to \infty$. In words, the simplex is approximated by a point.

The previous three theorems concern the empty corner distance. To complete the picture, the following lemma shows that (with the proper normalization) the empty corner distance bounds from below the in-out distance. So, the inapproximability results also carry on to the in-out distance. 

\begin{lemma}
\label{lem:IOtoE}
If $L \subset  \R^n$ is a convex body such that $\lambda_{\Delta,o}(L) = 1$ then
\[
D_e(L;\Delta^n) \leq n \cdot D_{io}(L;\Delta^n).
\]
\end{lemma}

\begin{lemma}
\label{lem:IOe-third}
If $L \subset \R^n$ is a convex body such that $\lambda_{\Delta,o}(L) = 1$ then
\[
D_e(L;\Delta^n) \geq D_{io}(L;\Delta^n).
\]
\end{lemma}
Both lemmas are sharp
(see remarks after their proofs). Somewhat surprisingly, the proof of \Cref{lem:IOe-third}
relies on a Kadets-type result proved by Akopyan and Karasev~\cite{akopyan2012kadets}.

\subsection{A characterization of simplices}

The first three theorems stated above are based on the fact that the simplex satisfies the following additivity property.

\begin{defn}
A convex body $K \subset \R^n$ is called outer additive if for every $L = L_1+L_2$ where $L_1,L_2 \subset \R^n$ are convex bodies, it holds that
\[\lambda_{K,o}(L) = \lambda_{K,o}(L_1) + \lambda_{K,o}(L_2).\]
\end{defn}

The additivity of $K$ does not rely on the ``position'' of $K$; it is invariant under invertible affine maps.
Because it is always true that
$\lambda_{K,o}(L) \leq \lambda_{K,o}(L_1) + \lambda_{K,o}(L_2)$, the additivity of $K$ is equivalent to 
$\lambda_{K,o}(L) \geq \lambda_{K,o}(L_1) + \lambda_{K,o}(L_2)$. 

The fact that simplices are outer additive easily follows from the additivity of support functions.
Are simplices the only outer additive convex bodies?

\begin{theorem}
\label{thm:AddIsSimplex}
A convex body is outer additive iff it is a simplex.
\end{theorem}

\begin{remark*}
The theorem is true even when the convex bodies $L_1,L_2$ in the definition of outer additivity
are restricted to be intervals. 
\end{remark*}

\subsection{Roadmap} First, \Cref{sec:outAdd} concerns outer additivity (\Cref{thm:AddIsSimplex}).
Second, \Cref{sec:IO} is about the in-out distance (\Cref{thm:Dio}). Third, \Cref{sec:De} treats the empty corner distance 
(\Cref{thm:De}, \Cref{thm:DeUnion}, \Cref{lem:DeSym} and \Cref{thm:Simplex} are proved in \Cref{sec:E-bounds}).
Finally, in \Cref{sec:relateD} we relate the two distances (\Cref{lem:IOtoE} is proved in \Cref{sec:IOtoE},
and \Cref{lem:IOe-third} is proved in \Cref{sec:IOtoE-two}).

\section{Outer additivity}
\label{sec:outAdd}

Fix $n$ and let $\Delta = \Delta^{n}$.
Denote the vertices of $\Delta$ by $v_1,\ldots,v_{n+1}$.
For each $j \in [n+1]$, let $F_j$ be the hyperplane that contains the facet that is opposite to~$v_j$. Call $F_1,\ldots,F_{n+1}$ the facet hyperplanes of $\Delta$. Let $u_j$ be the unit vector and $b_j \in \R$ be such that
$F_j  = \{x\in\mathbb{R}^n:\ip{u_j}{x} = b_j\}$
and $\ip{u_j}{v_j} > b_j$.

For $u \in \S^{n-1}$ and a convex body $K \subset \R^n$, denote by $\mathcal{S}(K,u)$
the supporting hyperplane of $K$ in direction $u$, which is of the form 
\[\mathcal{S}(K,u)= u^\perp + h_K(u) u.\]

\subsection{Simplices are outer additive}

The simplex is outer additive due to the following well-known additivity of support functions: 
$h_{L_1+L_2} = h_{L_1} + h_{L_2}$.

\begin{claim}
\label{clm:SimAdd}
All simplices are outer additive.
\end{claim}

\begin{proof}
Let $L = L_1+L_2$ for two convex bodies $L_1,L_2$.
The simplex $\Delta = \Delta^n$ has $n+1$ facets with unit normals $u_1,\ldots,u_{n+1}$. 
The $n+1$ supporting hyperplanes $\mathcal{S}(L,-u_1), \ldots,\mathcal{S}(L,-u_{n+1})$ define a homothet of $\Delta$ that contains $L$, and this homothet is exactly $\Delta_o(L)$.
The same is true for each of $L_1,L_2$.
By the additivity of support functions, 
\begin{equation}
\label{eqn:DeltaO}
\Delta_{o}(L)
= \Delta_{o}(L_1)+\Delta_{o}(L_2).
\qedhere    
\end{equation}
\end{proof}

\subsection{Outer additive bodies are simplices}

For $u \in \S^{n-1}$ and a convex body $K \subset \R^n$,
let
\begin{align*}
A_u = A_u(K)&= \mathcal{S}(K,-u) \cap K
\intertext{and}
B_u = B_u(K)&= \mathcal{S}(K,u) \cap K.
\end{align*}
For $x,y \in \R^n$, denote by $[x,y]$
the interval between $x,y$.
The main property we use is that
\begin{align}
\label{eq:segment}
a \in A_u, \ b \in B_u \quad \Longrightarrow \quad
\lambda_{K,o}([a,b]) = 1.
\end{align}
For $x \in \partial K$,
denote by $\Phi(x) = \Phi_K(x)$ the collection of $y \in \partial K$ such that 
$[x,y] \subset \partial K$.

The key step is the following geometric lemma that gives a necessary ``local'' condition for being outer additive. 
\begin{lemma}
\label{lem:notAdditive}
Let $K \subset \R^n$ be a convex body of dimension $n$.
If there are $u \in \S^{n-1}$, $a \in A_u$, $b \in B_u$ and $x \in \partial K \setminus (\Phi(a) \cup \Phi(b))$ then $K$ is not outer additive.
\end{lemma}

\begin{proof}
Suppose the condition is satisfied with $u,a,b,x$. 
Let $u' \in \S^{n-1}$ be such that $x\in A_{u'}$
and let $y \in B_{u'}$. 
Because $a, x, y \in K$, 
\[[ \tfrac{a+y}{2}, \tfrac{a+x}{2}] \subset K.\] 
By choice of $x$, it holds that
$[a,x] \not \subseteq \partial K$. So, $[a,x]$ must contain an interior point of $K$ which implies  \[\tfrac{a+x}2 \not \in \partial K.\] 
We can conclude that there is $\alpha_1 >1$ such that
\[[ \tfrac{a+y}2, \tfrac{a+y}2 + \tfrac{\alpha_1(x-y)}2 ] \subset K .\]
By the same argument for the points $b, x, y$, 
there is $\alpha_2 > 1$ so that
\[[ \tfrac{b+y}{2}, \tfrac{b+y}{2} + \tfrac{\alpha_2(x-y)}2 ] \subset K.\]
Setting $\alpha = \min\{\alpha_1, \alpha_2\}$, we conclude
that the following 
parallelogram $L$ lies in $K$:
\[
L = \conv \{ 
\tfrac{a+y}2, \tfrac{a+y}2 + \tfrac{\alpha(x-y)}2, 
\tfrac{b+y}2, \tfrac{b+y}2 + \tfrac{\alpha(x-y)}2
\}.
\]
It can be expressed as $L = L_1 + L_2$ with
\[
L_1 = [ \tfrac{a+y}2, \tfrac{b+y}2 ] 
\text{\ \ \ \ and\ \ \ \ } 
L_2 = [ 0, \tfrac{\alpha(x-y)}2 ] .
\]
Because $L \subset K$, 
\[\lambda_{K,o}(L) \leq 1.\]
The interval $L_1$ is parallel to $[a,b]$. 
Because $a \in A_u$ and $b \in B_u$,
\[\lambda_{K,o}([a,b]) = 1,\]
so
\[\lambda_{K,o}(L_1) = \lambda_{K,o}(\tfrac{1}{2}[a,b]) = \tfrac{1}{2}.\]
By the same argument
applied to $L_2$,
\[\lambda_{K,o}(L_2) = \lambda_{K,o}(\tfrac{\alpha}{2}[x,y]) = \tfrac{\alpha}{2}.\]
Because $\alpha >1$, it follows that $K$ is not outer additive.
\end{proof}

\begin{proof}[Proof of \Cref{thm:AddIsSimplex}]
We already know that simplices are outer additive.
It remains to prove that outer additive convex bodies are simplices. 
The proof is by induction on~$n$.
The case $n=1$ is trivial.
For $n>1$, assume that $K$ is outer additive and of full dimension. 
Let $u \in \S^{n-1}$ be such that $B_u$ has maximum dimension
(the dimension of a set is the minimum dimension of an affine space containing it).
Let $a \in A_u$.

\begin{claim}
\label{clm:partialK}
If $\partial K \subseteq B_u \cup \Phi(a)$
then $K = \{a\} * B_u$.    
\end{claim}

\begin{proof}
Assume $\partial K \subseteq B_u \cup \Phi(a)$.
The inclusion $K \supseteq \{a\} * B_u$ is trivial. It remains to prove the other inclusion.
Let $p$ be a point in the interior of $K$.
Let $F$ be a two-dimensional plane containing $a$ and $p$. Denote by $L$ the two-dimensional convex body $L = F \cap K$, so  $\partial L = F \cap \partial K$. By assumption, $\partial L \subseteq (F \cap B_u) \cup (F \cap \Phi(a))$.
For every $x \in F \cap \Phi(a)$, we have
$[a,x] \subset \partial L$. So $F \cap \Phi(a)$ comprises two line segments. Consequently, $F \cap B_u$ is a closed interval of positive length, and $L$ is the triangle $\{a\} * (F \cap B_u) = F \cap (\{a\} * B_u)$.
So, we showed that $p \in \{a\} * B_u$ and consequently $K \subseteq \{a\} * B_u$.
\end{proof}

\begin{claim}
\label{clm:dimBu}
The dimension of $B_u$ is $n-1$.
\end{claim}

\begin{proof}
If $\partial K \subseteq B_u \cup \Phi(a)$
then by \Cref{clm:partialK} we know that $K = \{a\} * B_u$ which implies the claim because $K$ has full dimension. 
Otherwise, there is $x \in \partial K$ such that
$x \not \in B_u \cup \Phi(a)$.
Because $K$ is outer additive, by \Cref{lem:notAdditive}, for all $b \in B_u$,
we know that $x \in \Phi(b)$.
Assume towards a contradiction that the dimension of $B_u$ is smaller than $n-1$.
Let $F$ be a hyperplane so that $x \in F$ and $B_u \subseteq F$.
It follows that $F$ is a supporting hyperplane of $K$ and that $F \cap K$ has dimension larger than that of $B_u$ which is a contradiction. 
\end{proof}

\begin{claim}
If $b$ is in the relative interior of $B_u$ then $\Phi(b) = B_u$.
\end{claim}

\begin{proof}
The inclusion $B_u \subseteq \Phi(b)$ is trivial.
It remains to prove the other inclusion. 
Assume towards a contradiction that $x \in \Phi(b) \setminus B_u$ so that $[b,x] \subset \partial K$. Because $b$ is in the relative interior of $B_u$ and by \Cref{clm:dimBu},
every continuous path on $\partial K$ from $b$
to outside $B_u$ has an open part inside $B_u$. In particular,
there is $y$ in the interior of $[b,x]$ such that $[b,y] \subset B_u \subset \mathcal{S}(K,u)$. This implies that 
$[b,x] \subset \mathcal{S}(K,u) \cap K = B_u$, which is a contradiction.
\end{proof}

Let $b$ be in the relative interior of $B_u$ so that $\Phi(b) = B_u$.
By \Cref{lem:notAdditive}, because $K$ is outer additive,
\[\partial K = \Phi(a) \cup \Phi(b) = \Phi(a) \cup B_u.\]
By \Cref{clm:partialK},
\[K = \{a\} * B_u.\]
It follows that for every convex body $L \subset \mathcal{S}(K,u)$, we have $\lambda_{K,o}(L) = \lambda_{B_u,o}(L)$.
Hence, $B_u$ is outer additive in $\mathcal{S}(K,u) \cong \R^{n-1}$ because for every convex bodies $L_1,L_2 \subset \mathcal{S}(K,u)$,  \begin{align*}
\lambda_{B_u,o}(L_1+L_2) 
& = \lambda_{K,o}(L_1+L_2) \\
& = \lambda_{K,o}(L_1)+\lambda_{K,o}(L_2) \\
& = \lambda_{B_u,o}(L_1)+\lambda_{B_u,o}(L_2).
\end{align*}
By induction, $B_u$ is an $(n-1)$-dimensional simplex, so $K$ is an $n$-dimensional simplex.
\end{proof}

\section{The in-out distance}
\label{sec:IO}

\begin{proof}[Proof of \Cref{thm:Dio}]
It suffices to consider $L = L_1+L_2$ where $L_1,L_2$ are convex bodies
and prove that
$$D_{io}(L;\Delta) \geq D_{io}(L_1;\Delta)$$
where $\Delta = \Delta^n$.
By translating $L_1$ if needed,
we can assume (without loss of generality) that
\begin{align*}
L_2 \subseteq \Delta_{o}(L_2) \subseteq \frac{\lambda_{\Delta,o}(L_2)}{\lambda_{\Delta,i}(L_1)} L_1 .    
\end{align*}
Consequently,
\begin{align*}
L_1 + L_2 \subseteq \left( 1 + \frac{\lambda_{\Delta,o}(L_2)}{\lambda_{\Delta,i}(L_1)} \right) L_1 
\end{align*}
so that
\begin{align*}
\lambda_{\Delta,i}(L_1+L_2)
& \leq \left( 1 + \frac{\lambda_{\Delta,o}(L_2)}{\lambda_{\Delta,i}(L_1)} \right) \lambda_{\Delta,i}(L_1) \\
& = \lambda_{\Delta,i}(L_1) + \lambda_{\Delta,o}(L_2).
\end{align*}
Combined with \Cref{clm:SimAdd}, 
\begin{align*}
D_{io}(L_1+L_2;\Delta)
& = \lambda_{\Delta,o}(L_1+L_2)
- \lambda_{\Delta,i}(L_1+L_2) \\
& \geq \lambda_{\Delta,o}(L_1) + \lambda_{\Delta,o}(L_2)
-\lambda_{\Delta,i}(L_1) - \lambda_{\Delta,o}(L_2) \\
& = D_{io}(L_1;\Delta).
\qedhere 
\end{align*}

\end{proof}

\begin{remark*}
The only property of $\Delta$ the proof above relies on is outer additivity (which we 
a posteriori know to be equivalent to being a simplex).
\end{remark*}

\section{The empty corner distance}
\label{sec:De}

\subsection{Barycentric coordinates}

For hyperplane $F  = \{x\in\mathbb{R}^n:\langle u,x\rangle = b\}$ with a unit vector $u$, define the oriented distance from $p \in \R^n$ to $F$ by
\[
\dist(p,F) := \ip{u}{p} - b.
\]
The sign of the oriented distance depends on the specific representation of $F$ given by $u,b$.

\begin{lemma}
\label{lem:um-ratios}
Let $v_1,\ldots,v_{n+1}$ be the vertices and $F_1,\ldots,F_{n+1}$ be the facet hyperplanes of $\Delta^{n}$.
For every $p \in \R^n$,
\[
\sum_{j \in [n+1]} \frac{\dist(p,F_j)}{\dist(v_j,F_j)} = 1.
\]
\end{lemma}

\begin{proof}	
Every point \(p \in \R^n\) can be uniquely expressed in barycentric coordinate system (see e.g.~\cite[Section~1]{rockafellar1997convex}):
\begin{align}
\label{eq-bar}
p  = \sum_j \alpha_j  v_j
\ \ \ \text{where}\ \ \ 
\sum_j \alpha_j = 1.
\end{align}
We shall prove that for every $j$,
\[
\alpha_j = \frac{\dist(p,F_j)}{\dist(v_j,F_j)}.
\]
Fix $j$ and take the inner product with $u_j$.
Because $v_j \in F_k$ for $k \not = j$,
\begin{align*}
\langle u_j,p\rangle
&= \sum_k \alpha_k \langle u_j,v_k \rangle \\ &=\alpha_j  \langle u_j,v_j\rangle
+\sum_{k \neq j} \alpha_k b_j 
\\ &= \alpha_j  \langle u_j,v_j\rangle
 + (1-\alpha_j)  b_j
\\ &= b_j  + \alpha_j  \bigl(\langle u_j,v_j\rangle - b_j\bigr).
\end{align*}
Hence,
\[
\alpha_j
=
\frac{\langle u_j,p\rangle - b_j}{\langle u_j,v_j\rangle - b_j}
=
\frac{\dist(p,F_j)}{\dist(v_j,F_j)}.
\qedhere
\]
\end{proof}

\begin{remark*}
Using the barycentric coordinate system, we can think of the empty corners of a convex body $L \subset \R^n$ as follows.
Using the notation from \Cref{lem:um-ratios},
for a point $p \in L$, write $p$ in barycentric coordinates as $(p_v : v \in V(\Delta))$ where
\[p_v = \frac{\dist(p,F_j)}{\dist(v,F_j)}.\]
The (relative) size of the empty corner with respect to $v$ is
\[\frac{E_v(L;\Delta)}{\lambda_{\Delta,o}(L)}
= \inf \{1 - p_v : p \in L\}
= 1 - \sup \{p_v : p \in L\}.\]    
\end{remark*}

\subsection{Bounding the distance}
\label{sec:E-bounds}

\begin{claim}
\label{clm:MinSum}
    Let $\Delta = \Delta^n$.
	Let $L=L_1+L_2$ where $L_1,L_2 \subset \R^n$ are convex bodies. 
    Then, for all $v \in V(\Delta)$, \[E_v(L,\Delta_o(L)) = E_v(L_1,\Delta_o(L_1)) + E_v(L_2,\Delta_o(L_2)).\]
If in addition $D_e(L_1;\Delta) \leq D_e(L_2;\Delta)$ then 
    \[D_e(L_1;\Delta) \leq D_e(L;\Delta) \leq D_e(L_2;\Delta).\]

\end{claim} 
\begin{proof}
Fix a vertex $v_j \in V(\Delta)$.
A similar reasoning to the proof of \Cref{clm:SimAdd} implies the first part.
Let $u_1,\ldots,u_{n+1}$ be the normals to the facet hyperplanes of~$\Delta$.
The outer simplex $\Delta_o(L)$ is constructed using the supporting hyperplanes $\mathcal{S}(L,-u_1),\ldots,\mathcal{S}(L,-u_{n+1})$.
Let $G \in \cH_\Delta$ be the homothet of $\Delta_o(L)$ with center of homothety $v_j$
such that $\lambda_\Delta(G) = E_{v_j}(L;\Delta)$ is the size of the empty corner of $v_j$.
The simplex $G$ has exactly the same facets as $\Delta_o(L)$,
except that the facet given by $\mathcal{S}(L,-u_{j})$ in $\Delta_o(L)$ is replaced
by the facet given by $\mathcal{S}(L,u_{j})$ in $G$.
The same is true for $L_1$ and $L_2$.
The additivity of support functions now implies the first part. 

The second part follows because summing over $v$, 
\[E(L;\Delta_o(L)) = E(L_1;\Delta_o(L_1)) + E(L_2;\Delta_o(L_2))\]
which implies
\begin{align*}
D_e(L;\Delta) & = \frac{E(L;\Delta_o(L))}{\lambda_{\Delta,o}(L)} \\
& = \frac{E(L_1;\Delta_o(L_1)) + E(L_2;\Delta_o(L_2))}{\lambda_{\Delta,o}(L_1) + \lambda_{\Delta,o}(L_2)} .
\end{align*}
Finally, use the fact that if $\tfrac{a}{b} \leq \tfrac{c}{d}$
then $\tfrac{a}{b} \leq \tfrac{a+c}{b+d} \leq \tfrac{c}{d}$.
\end{proof}

\begin{proof}[Proof of \Cref{thm:De}]
\Cref{clm:MinSum} implies
the theorem for $m=2$.
The general case follows by induction.
\end{proof}

\begin{proof}[Proof of \Cref{lem:DeSym}]
Denote by $v_1,\ldots,v_{n+1}$ the vertices of $\Delta = \Delta^{n}$.
Assume (without loss of generality) that $L$ is symmetric with respect to the origin $p$.
Let $F_1,\ldots,F_{n+1}$ be the facet hyperplanes of $\Delta_o(L)$, where
$F_j  = \{x\in \R^n : \langle u_j,x\rangle = b_j\}
$ with $\ip{u_j}{v_j}>0$.

Start by fixing $j$.
Each $F_j$ is tangent to $L$ at a point $y_j$.
Denote by $F'_j$ the hyperplane
\[F'_j=\{x \in \R^n : \ip{-u_j}{x}  = b_j \}.\]
Because $L$ is symmetric, $F'_j$ is tangent to $L$ at the point
$-y_j$. 
In addition,  
\[\dist(p,F_j) = - b_j = \dist(p,F'_j)\]
and
\begin{align*}
\dist(v_j,F_j)
& = \ip{u_j}{v_j} - b_j \\
& =  \ip{u_j}{v_j} + b_j - 2 b_j \\
& = - \dist(v_j,F'_j) +  2\cdot \dist(p,F_j) .
\end{align*}
The hyperplane $F'_j$ together with the hyperplanes $\{F_k\}_{k \neq j}$ define the ``empty corner simplex of $v_j$'' so that
\begin{align}
\frac{E_{v_j}(L;\Delta_o(L))}{\lambda_{\Delta,o}(L)} = - \frac{\dist(v_j,F'_j)}{\dist(v_j,F_j)} \geq 0.
\label{eqn:EmptyCor}
\end{align}

Now, sum over $j$.
Applying \Cref{lem:um-ratios} to the origin $p$,
\[
\sum_{j} \frac{\dist(p,F_j)}{\dist(v_j,F_j)} = 1.
\]
Hence,
\begin{align*} 
D_e(L)
&= -\sum_j \frac{\dist(v_j,F'_j)}{\dist(v_j,F_j)}
\\ &= \sum_j \frac{\dist(v_j,F_j) - 2\cdot \dist(p,F_j)}{\dist(v_j,F_j)}
\\ &= n+1 - 2 \sum_j \frac{\dist(p,F_j)}{\dist(v_j,F_j)}
= n-1.
\qedhere
\end{align*}
\end{proof}

\begin{lemma}  
\label{lem:DeInf}
Let $\Delta = \Delta^n$.
Let $L \subset \R^n$ be a convex body.
Then, there is $m =  m(\Delta,L) > 0$ so that
for all $H' \in \cH_\Delta$ such that $H' \supseteq L$,
\[ \frac{E(L;H')}{\lambda_\Delta(H')} = 
D_e(L;\Delta) + m \left( 1 - \frac{\lambda_{\Delta,o}(L)}{\lambda_\Delta(H')} \right)\]
and consequently
\[ \frac{E(L;H')}{\lambda_\Delta(H')} \geq D_e(L;\Delta).\]

\end{lemma}

\begin{proof}
Let $H' \in \cH_\Delta$ be such that $H' \supseteq L$.
By translation if needed, we can assume without loss of generality that $H' = k \Delta_o(L)$
where $k \geq 1$.
In other words, the center of homothety between $H'$ and $\Delta_o(L)$ is the origin $p$.

Denote the vertices of $\Delta_o(L)$ by
$v_1, \ldots, v_{n+1} \in \R^n$, and denote the corresponding vertices of $H'$ by $v'_1,\ldots,v'_{n+1} \in \R^n$ so that
$v'_j = k \cdot v_j$.
For each $j$,
let $F_j = \{x\in\mathbb{R}^n:\langle u_j,x\rangle = b_j\}$ be the facet hyperplane of $\Delta_o(L)$ opposite to $v_j$.
Let $F'_j  = \{x\in\mathbb{R}^n:\langle u_j ,x\rangle = b'_j\}$ be the facet hyperplane of $H'$ opposite to $v'_j$ such that $b'_j = k \cdot b_j$. Let $G_j = \{x\in\mathbb{R}^n:\langle  u_j ,x\rangle = c_j\}$ be the supporting hyperplane to $L$ in the direction $u_j$.
It follows that
\begin{align*}
\dist(v'_j, F'_j)
& = \ip{u_j}{v'_j} - b'_j \\
& = k \cdot (\ip{u_j}{v_j} - b_j) \\
& = k \cdot \dist(v_j, F_j) 
\end{align*}
and
\begin{align*}
\dist(p,G_j) - \dist(v'_j,G_j)
& = c_j - \ip{u_j}{v'_j} + c_j \\
& = c_j - \ip{u_j}{k \cdot v_j} + c_j \\
& = k \cdot (\dist(p,G_j) - \dist(v_j,G_j)).
\end{align*}
Similarly to \Cref{eqn:EmptyCor},
\begin{align*}
\frac{E(L;H')}{\lambda_\Delta(H')}
& = \sum_j \frac{\dist(v'_j,G_j)}{\dist(v'_j,F'_j)} 
\\ 
& = \sum_j \frac{\dist(p,G_j)
-  k \cdot (\dist(p,G_j) - \dist(v_j,G_j))}{k \cdot \dist(v_j, F_j)}  \\ 
& = \left( \sum_j \frac{\dist(v_j,G_j)}{ \dist(v_j, F_j)} \right)
 + \left( \frac{1-k}{k} \sum_j \frac{\dist(p,G_j)
}{\dist(v_j, F_j)} \right)
\\ 
& = D_e(L)
 +  \frac{1-k}{k} \sum_j \frac{\dist(p,G_j)
}{\dist(v_j, F_j)}  .
\end{align*}

The hyperplanes $F_1,\ldots,F_{n+1}$ define
$n+1$ halfspaces such that their intersection is $\Delta_o(L)$.
Similarly, the hyperplanes $G_1,\ldots,G_{n+1}$ define halfspaces such that their intersection is a simplex $N$.
The simplex $N$ is of the form $x - m H$ where $m  = m(L) > 0$ (it is a ``negative homothet'' of $\Delta$).
Denote its vertices by $w_1,\ldots,w_{n+1}$
where $w_j$ corresponds to $v_j$.
It follows that for all $j$,
\[\dist(w_j, G_j) = - m \cdot \dist(v_j,F_j).\]
Apply \Cref{lem:um-ratios} to the origin $p$ and the simplex $N$,
\begin{align*}
\sum_j \frac{\dist(p,G_j)}{\dist(v_j,F_j)} =
- m \cdot \sum_j \frac{\dist(p,G_j)}{\dist(w_j,G_j)} = - m < 0.
\end{align*}
Finally, 
\begin{equation}
\label{eq:f-increasing}
\frac{E(L;H')}{\lambda_\Delta(H')} = 
D_e(L;\Delta) + m \left( 1 - \frac{\lambda_{\Delta,o}(L)}{\lambda_\Delta(H')} \right).
\qedhere    
\end{equation}

\end{proof}

\begin{proof}[Proof of \Cref{thm:DeUnion}]
\label{proof:thm:DeUnion}
Let $H = \Delta_o(L)$.
For every vertex $v$ of $H$,
\begin{align*}
E_{v}(L;H) 
& = \min \{ E_{v}(L_1;H), E_{v}(L_2;H)\} \\
&\geq  E_{v}(L_1,H) + E_{v}(L_2;H) - \lambda_{\Delta}(H) .
\end{align*}
Thus, summing over $v$ and using \Cref{lem:DeInf},
\begin{align*} 
D_e(L) &
\geq -(n+1) + \frac{E(L_1;H)}{\lambda_{\Delta}(H)} + \frac{E(L_2;H)}{\lambda_{\Delta}(H)} \\
& \geq -(n+1) + D_e(L_1) + D_e(L_2). \qedhere
\end{align*}
\end{proof}

\begin{proof}[Proof of \Cref{thm:Simplex}]
The proof of the $\leq$ side is by construction.
In depth $d$, we can trivially build $\Delta^{2^d}$ which sits inside $\Delta^n$.
For every $v \in V(\Delta^{2^d})$, 
the size of the empty corner of $v$ is zero.
For every other vertex of $\Delta^n$, the size is one.
It follows that
\[D_e(\cP_{n,d}) \leq n+1- 2^d.\]
The proof of the $\geq$ side is by induction. 
The induction step is $d=1$.
Every $P \in \cP_{n,1}$ is a zonotope and hence centrally symmetric.
The claim follows by \Cref{lem:DeSym}:
\[D_e(\cP_{n,1}) = n-1 = n+1 - 2^d .\]
The induction step follows from \Cref{thm:De} and \Cref{thm:DeUnion}.
Assume $P = \sum_j (K_j*L_j) \not \in \cP_{n,0}$ where $K_j,L_j \in \cP_{n,d-1}$ with $d>1$.
We can assume that $K_j *L_j \not \in \cP_{n,0}$ for all $j$. 
Hence, 
\begin{align*}
D_e(P) 
& = n+1 - \coD_e(P) \\
& \geq n+1 - \text{max}_j \coD_e(K_j*L_j) \\
& \geq n+1 - 2^{d-1}-2^{d-1} ;
\end{align*}
the first inequality follows from \Cref{thm:De} 
and the second from \Cref{thm:DeUnion}.
To apply \Cref{thm:DeUnion}, we need
both $K_j,L_j$ to be convex bodies;
if e.g.\ $K_j \in \cP_{n,0}$ then we can replace $K_j$ by $K_j *\{x\}$ for $x \in L_j \setminus K_j$ and still apply induction.
\end{proof}

\section{Relating the two distances}
\label{sec:relateD}
The goal of this section is to relate the empty corner distance and the in-out distance
(Lemmas~\ref{lem:IOtoE} and~\ref{lem:IOe-third}).

\subsection{Part one}
\label{sec:IOtoE}

The following claim is an extension of \Cref{lem:IOtoE}.

\begin{claim}
\label{clm:IOe}
Let $\Delta = \Delta^n$. Let $M \subset \R^n$ be a convex body.
Then,
\[
D_e(M) \leq n \left(1 - \frac{\lambda_{\Delta,i}(M)}{\lambda_{\Delta,o}(M)}\right).
\]
\end{claim}
\begin{proof}
Let 
$L \in \cH_\Delta$ be the simplex such that $L \subseteq M$ and $\lambda_\Delta(L) = \lambda_{\Delta,i}(M)$.
Because $L \subseteq M$,
\[D_e(M) \leq \frac{E(L; \Delta_o(M))}{\lambda_{\Delta,o}(M)}.\]
It suffices to prove that
\[ \frac{E(L; \Delta_o(M))}{\lambda_{\Delta,o}(M)} = n \left(1 - \frac{\lambda_{\Delta,i}(M)}{\lambda_{\Delta,o}(M)} \right).\]
For this, we use \Cref{eq:f-increasing}.
As the simplex $H'$ from \Cref{lem:DeInf} we take $\Delta_o(M)$:
\[
\frac{E(L; \Delta_o(M))}{\lambda_{\Delta,o}(M)}
= D_e(L;\Delta)
+  \frac{k-1}{k} m = \frac{k-1}{k} m,
\]
where \[k = \frac{\lambda_{\Delta,o}(M)}{\lambda_{\Delta,i}(M)}\] 
and $(-m)$ is the negative coefficient of homothety between the simplex $L$ and its negative homothet $N$ in which $L$ is inscribed.
It follows that $m=n$ so that
\[
\frac{k-1}{k} m
= n \left (1 - \frac{\lambda_{\Delta,i}(M)}{\lambda_{\Delta,o}(M)} \right).
\qedhere 
\]
\end{proof}

\begin{remark*} 
Equality is achieved in \Cref{clm:IOe} when the inner simplex $L$ ``touches all the empty corners'' of $M$. It happens in many cases, including $M=\Delta$ and $M=-\Delta$.
\end{remark*}

\subsection{Part two}
\label{sec:IOtoE-two}
The following claim is an extension of \Cref{lem:IOe-third}.

\begin{claim}
\label{clm:IOe-second}
For every convex body $L \subset \R^n$,
\[
D_e(L) +  \frac{\lambda_{\Delta,i}(L)}{\lambda_{\Delta,o}(L)} \geq 1.
\]
\end{claim}

Before proving the claim, we need some preparation.
A supporting half-space $T$ of a convex body $L \subseteq \R^n$ is of the form
\[ T= \{x \in \R^n : \ip{u}{x} \leq h_L(u) \}\]
for some $u \in \S^{n-1}$.
It always holds that $L \subseteq T$.

\begin{lemma}
\label{lem:inscribed}
Let $L\subset\mathbb{R}^n$ be a convex body
with a non-empty interior.
Then, there exist $k \leq n+1$ supporting half-spaces $T_1,\ldots,T_k$ of $L$ such that
\[
L \subseteq A := \bigcap_{j \in [k]} T_j
\quad 
\text{and}
\quad 
\lambda_{\Delta,i}(L) = \lambda_{\Delta,i}(A).
\]
\end{lemma}

The proof of the lemma uses techniques from \cite{klain2010containment}.
In the proof, we use that for every two convex bodies $K,L \subset \R^n$,
\[
K \subseteq L
\quad\Longleftrightarrow\quad
h_K(u) \leq h_L(u) 
\quad\forall u\in \S^{n-1}.
\]
For a convex body $L \subset \R^n$,
we denote by $\Delta_i(L)$
a simplex $H \in \cH_{\Delta^n}$ such that
 $H \subseteq L$ and $\lambda_{\Delta^n}(H) = \lambda_{\Delta^n,i}(L)$
 (it is not necessarily unique).
 If $L$ has non-empty interior,
 then $\Delta_i(L)$ has non-empty interior.

\begin{proof}[Proof of \Cref{lem:inscribed}]
Let $\Delta_i = \Delta_i(L)$.
Because $\Delta_i \subseteq L$,
\[
h_L(u) \geq h_{\Delta_i}(u) 
\quad\forall u\in \S^{n-1}.
\]
Let
\[
U = \bigl\{ u\in \S^{n-1} : h_L(u) = h_{\Delta_i}(u)  \bigr\}.
\]

We claim that $0\in \conv(U)$. Otherwise, because $U$ is closed, there exist $\eps>0$ and $v\in \S^{n-1}$ such that
$\langle v,u\rangle < - \eps$ for all $u\in U$.  
In other words, $U$ is contained in the (open) half-space $W:=\{w \in \R^n : \ip{v}{w} < -\eps\}$.
The set $\S^{n-1} \setminus W$ is compact and is disjoint from $U$ so there is $\eps' > 0$ such that
\begin{align*}
h_L(u) \geq h_{\Delta_i}(u) + \eps'
\quad\forall u\in \S^{n-1} \setminus W.
\end{align*}
Hence, there is $\delta>0$ such that
\begin{align*}
h_L(u) > h_{\Delta_i}(u) + \delta\langle v,u\rangle = h_{\Delta_i + \delta v}(u)
\quad\forall u\in \S^{n-1}.
\end{align*}
It follows that $\Delta_i + \delta v$ is contained in the interior of $L$, and therefore $L$ contains a larger homothet of~$\Delta$, which is a contradiction. 

Now, by Carath\'eodory's  theorem, there exist $u_1,\dots,u_k \in U$
and $\alpha_1,\ldots,\alpha_k > 0$ for
$2 \leq k\leq n+1$ such that 
\[
\sum_{j} \alpha_j u_j = 0.
\]
Define the supporting half-spaces as ($j \in [k]$)
\[
T_j
=\{ x : \langle x,u_j\rangle \leq h_L(u_j)\}
\]
and their intersection
\[
A=\bigcap_j T_j \supseteq L.
\]

It remains to prove that
$\lambda_{\Delta,i}(L) = \lambda_{\Delta,i}(A)$.
Because $L \subseteq A$,
we have $\lambda_{\Delta,i}(L) \leq \lambda_{\Delta,i}(A)$.
Assume towards a contradiction
that
\[(1+\gamma) \lambda_{\Delta,i}(L) \leq  \lambda_{\Delta,i}(A)\]
for some $\gamma > 0$ so that 
\[(1+\gamma) \Delta_i + t \subseteq A\]
for some $t \in \R^n$.
For every $j \in [k]$, 
\[
(1+\gamma) h_{\Delta_i}(u_j) + \ip{t}{u_j}
= h_{(1+\gamma) \Delta_i+t}(u_j)
\leq h_A(u_j)
= h_L(u_j) = h_{\Delta_i}(u_j)
\]
so 
\[
\gamma h_{\Delta_i}(u_j) + \ip{t}{u_j} \leq 0 .
\]
Because $h_{\Delta_i}$ is strictly convex 
at the origin, we get a contradiction:
\[
0 = \sum_j \alpha_j \ip{t}{u_j}
\leq -\gamma \sum_j \alpha_j h_{\Delta_i}(u_j)
< -\gamma h_{\Delta_i}\left( \sum_j \alpha_j u_j \right) = 0. \qedhere
\]
\end{proof}

We also use the following Kadets-type result by Akopyan and Karasev~\cite{akopyan2012kadets}.
It is based on the notion of inductive covering (which is defined by induction as follows).
The cover of $\R^n$ by a single set $V_1 = \R^n$ is inductive.
A cover $\R^d = V_1\cup\cdots\cup V_k$ is inductive if $V_1,\ldots,V_k$ are closed convex sets and for each $j \in [k]$, there exists an inductive covering
	\[
	\R^d = W_1\cup\cdots\cup W_{i-1}\cup W_{i+1}\cup\cdots\cup W_k 
	\]
	such that
\[W_\ell \subseteq  V_\ell \cup V_j
\qquad  \forall \ell \neq j.\]
The following\footnote{In \cite{akopyan2012kadets} the theorem is stated for partitions instead of coverings. The remark before Definition~5 states that the theorem holds for coverings as well. In addition, their notation is different from ours: a convex body means with a non-empty interior,
and $\lambda_{B,i}(C)$ is denoted by $r_B(C)$.}
 is Theorem 1 in \cite{akopyan2012kadets}.

\begin{theorem}
\label{thm:inductive-kadets}
Let $V_1,\ldots,V_k$ be an inductive covering of $\R^n$.
	Let $B\subset\mathbb{R}^n$ be a convex body with a non-empty interior, and let $C_j = V_j\cap B$ for each~$j \in [k]$. Then,
	\[
	\sum_j \lambda_{B,i}(C_j) \ge 1.
	\]
\end{theorem}

\begin{proof}[Proof of \Cref{clm:IOe-second}]
We assume without loss of generality that $L$ is 
smooth, strictly convex and has a non-empty interior, 
because the collection of such convex bodies is dense. 
We also assume without loss of generality that $\lambda_{\Delta,o}(L) = 1$ and prove that $D_e(L) + \lambda_{\Delta,i}(L) \geq 1.$

Let $\Delta = \Delta^n$. 
Let $\Delta_o = \Delta_o(L)$ and $\Delta_i = \Delta_i(L)$.
Consider the set
\[
U = \bigl\{ u\in \S^{n-1} : h_L(u) = h_{\Delta_i}(u)  \bigr\}
\] from the proof of \Cref{lem:inscribed}.
Because $L$ is strictly convex and $\Delta_i \subseteq L$, for every $u \in U$, the face $\Delta_i \cap \mathcal{S}(\Delta_i,u)$ of $\Delta_i$ is a vertex which we denote by $g(u) \in V(\Delta_i)$ (otherwise, the boundary of $L$ contains a line segment).
The vector $u$ is a normal to $L$ at $g(u)$. Because $L$ is strictly convex and smooth, the map $u \mapsto g(u)$ is injective. 
Choose the directions $u_j$ and the half-spaces $T_j$ (for $j \in [k]$) from the proof of \Cref{lem:inscribed}. 
Let 
\[A'_0 = \bigcap_{j} T_j\]
and 
\[A'_j = cl( \R^n \setminus T_j)\]
where $cl$ denotes closure. 
Let $A_j = A'_j \cap \Delta_o$ for $j \in \{0,1,\ldots,k\}$.
It follows that 
\begin{align*}
\R^n &= \bigcup_j A'_j
\intertext{and}
\Delta_o &= \bigcup_j A_j.
\end{align*}
By \Cref{lem:inscribed},
\[\lambda_{\Delta,i}(L) = \lambda_{\Delta,i}(A_0).\]

Fix $j \in [k]$ for now. The direction $u_j \in U$ defines the vertex $g_j := g(u_j)$ of $\Delta_i$.
Denote by $g'_j$ the vertex of $\Delta_o$
that corresponds to $g_j$: 
\[\{g'_j\} =  \mathcal{S}(\Delta_o,u_j) \cap \Delta_o.\]
We claim that 
\[\lambda_{\Delta,i}(A_j) \leq E_{g'_j}(L;\Delta).\]
Let $x \in \mathcal{S}(\Delta_o,-u_j) \cap \Delta_o$.
Because $\lambda_{\Delta,o}(L) = 1$ and the interval $[g'_j,x]$ is contained in $\Delta_o$,
\[
\lambda_{\Delta,i}(A_j) \leq \lambda_{[g'_j, x],i}(A_j) .
\]
The interval $[g'_j,x]$ intersects $\mathcal{S}(\Delta_i,u_j)$ at a point $y$. 
Using \eqref{eq:segment}, we can deduce
\[
\lambda_{[g'_j, x],i}(A_j) \leq \frac{|[y, x]|}{|[g'_j, x]|}.
\]
The homothety with center $g'_j$ that moves $x$ to $y$ also moves $\Delta_o$, which has
coefficient of homothety $\lambda_{\Delta,o}(\Delta_o)=1$, to a homothet $H_j$ with coefficient of homothety $\lambda_{\Delta,o}(H_j) = \frac{|[y, x]|}{|[g'_j, x]|}$. 
So,
\[\lambda_{\Delta,i}(A_j) \leq \lambda_\Delta(H_j).\]
Because the interiors of $A_j$ and $L$ do not intersect, 
\[E_{g'_j}(L;\Delta) \geq \lambda_\Delta(H_j).\]

Now, collect the information on all $j$'s. 
Because $u \mapsto g(u)$ is injective,
$g'_1,\ldots,g'_k$ are distinct. 
Summing over $j \in [k]$,
\[
D_e(L) \geq \sum_j \lambda_{\Delta,i}(A_j).
\]

Next, we show that we can use the Kadets-type result by Akopyan and Karasev.

\begin{claim}
\label{clm:A'isInd}
The covering $A'_0\cup A'_1\cup \dots \cup A'_k$ of $\R^n$ 
is inductive.
\end{claim}
\begin{proof}
We prove a more general claim.
For $J \subseteq [k]$, define the $J$-covering to be $V_0$ and $V_j$ for $j \in J$, where
\[V_0 = \bigcap_{j \in J} T_j\]
(if $J = \emptyset$ then $V_0 = \R^n$)
and for $j \in J$,
\[V_j = cl( \R^n \setminus T_j).\]
We prove by induction that the $J$-covering is an inductive covering. 
The $\emptyset$-covering is an inductive covering
(it has one set).
It remains to perform the inductive step.
Let $V_0$ and $V_j$ for $j \in J$ be the $J$-covering for $|J|>0$.
We need to verify that the relevant condition holds for every $j \in \{0\} \cup J$.

First, consider the case $j \neq 0$.
Consider the $(J\setminus \{j\})$-covering: 
\[W_0 = \bigcap_{\ell \in J\setminus \{j\}} T_\ell \]
and $W_\ell = V_\ell$ for $\ell \in J\setminus \{j\}$.
For $\ell = 0$, we use that
$V_0 = cl(W_0 \setminus V_{j})$ and get
\[
W_0 \subseteq cl(W_0 \setminus V_{j}) \cup V_{j} = V_0 \cup V_{j}.
\]
For $\ell \in J \setminus \{j\}$,
\[
W_\ell = V_\ell \subseteq V_\ell \cup V_{j}.
\]

Second, consider the case $j = 0$.
Choose some $i \in J$ and
consider the $(J \setminus \{i\})$-covering:
\[W_{i} = \bigcap_{\ell \in J\setminus \{i\}} T_\ell \]
and $W_\ell = V_\ell$ for 
$\ell \in J\setminus \{i\}$.
For $\ell = i$, we use that
$V_0 = cl(W_{i} \setminus V_{i})$ and get
\[W_{i} \subseteq cl(W_{i} \setminus V_{i}) \cup V_{i} = V_0 \cup V_{i}.
\]
For $\ell \in J \setminus \{i\}$,
\[
W_\ell = V_\ell \subseteq V_\ell \cup V_{0}.
\qedhere 
\]
\end{proof}

Putting it all together, by \Cref{clm:A'isInd} and \Cref{thm:inductive-kadets},
\[ 
\lambda_{\Delta,i}(L) + D_e(L) \geq \sum_{j \in \{0,1,\ldots,k\}} \lambda_{\Delta,i}(A_j) \geq 1.
\qedhere
\]

\end{proof}

\begin{remark*}
Equality is achieved in \Cref{clm:IOe-second} and 
\Cref{lem:IOe-third}
in the following example of $L$. Denote by $v_1,\ldots,v_{n+1}$ the vertices of $\Delta^{n}$.
Let $a$ and $b$ be two points on the edge $[v_1,v_2]$, 
and let $L$ be $\conv\{a,b,v_3,\dots,v_{n+1}\}$.
\end{remark*}

\bibliographystyle{amsplain}
\bibliography{mr.bib}

\end{document}